\documentclass[12pt,a4paper,twoside]{amsart}
\usepackage{amsmath,amsfonts,amsthm,amsopn,color,amssymb,enumitem}
\usepackage{palatino}
\usepackage{graphicx}
\usepackage[colorlinks=true]{hyperref}
\hypersetup{urlcolor=blue, citecolor=red, linkcolor=blue}
\usepackage[left=2.61cm,right=2.61cm,top=2.72cm,bottom=2.72cm]{geometry}

\usepackage[colorinlistoftodos]{todonotes}
\makeatletter
\providecommand\@dotsep{5}
\def\listtodoname{List of Todos}
\def\listoftodos{\@starttoc{tdo}\listtodoname}
\makeatother

\newcommand{\e}{\varepsilon}
\newcommand{\eps}{\varepsilon}

\newcommand{\R}{\mathbb{R}}

\newcommand{\RN}{{\mathbb{R}^N}}

\newcommand{\de}{\partial}

\renewcommand{\le}{\leslant}
\renewcommand{\ge}{\geslant}
\renewcommand{\a }{\alpha }

\renewcommand{\b }{\beta }

\renewcommand{\d }{\delta }

\newcommand{\g }{\gamma }

\renewcommand{\l }{\lambda}
\renewcommand{\ln }{\lambda_n}
\newcommand{\n }{\nabla }
\newcommand{\s }{\sigma }
\renewcommand{\t}{\theta}
\renewcommand{\O}{\Omega}

\newcommand{\Il}{I_{\l}}
\newcommand{\Iln}{I_{\l_n}}

\renewcommand{\H}{H^1(\RN)}
\newcommand{\Hr}{H^1_r(\RN)}

\newcommand{\Ne}{\mathcal{N}}

\newcommand{\N}{\mathbb{N}}

\newcommand{\bw}{\textbf{w}}

\newcommand{\irn }{\int_{\RN}}

\def\bbm[#1]{\mbox{\boldmath $#1$}}
\newcommand{\beq }{\begin{equation}}
\newcommand{\eeq }{\end{equation}}

\renewcommand{\le}{\leqslant}
\renewcommand{\ge}{\geqslant}
\newcommand{\dis}{\displaystyle}

\newcommand{\ird }

\newtheorem{theorem}{Theorem}[section]
\newtheorem{lemma}[theorem]{Lemma}

\newtheorem{proposition}[theorem]{Proposition}
\newtheorem{remark}[theorem]{Remark}
\newtheorem{corollary}[theorem]{Corollary}

\title[On Some Scalar Field Equations with Competing Coefficients]{On Some Scalar Field Equations with Competing Coefficients}

\author[G. Cerami]{Giovanna Cerami}
\address{Dipartimento di Meccanica, Matematica e Management,
Politecnico di Bari
\newline\indent
Via Orabona 4,  70125  Bari, Italy}
\email{giovanna.cerami@poliba.it}

\author[A. Pomponio]{Alessio Pomponio}
\address{Dipartimento di Meccanica, Matematica e Management,
Politecnico di Bari
\newline\indent
Via Orabona 4,  70125  Bari, Italy}
\email{alessio.pomponio@poliba.it}

\thanks{Work  supported by G.N.A.M.P.A. of I.N.D.A.M.
A. Pomponio is supported by G.N.A.M.P.A. Project ``Analisi variazionale di modelli fisici non lineari''}

\begin{document}

\maketitle

\begin{abstract}
This paper deals with semilinear elliptic problems of the type
\[ 
\left\{
        \begin{array}{ll}
 -\Delta u+\alpha(x)u= \beta (x)|u|^{p-1}u      \quad     \hbox{in }\R^N, 
\\
 u(x)>0\quad\hbox{in }  \R^N, \qquad u \in H^1(\R^N),         
        \end{array}
\right.
\]
 where $p$ is superlinear but subcritical and the coefficients $\alpha$ and $\beta$ are positive functions such that $\alpha(x) \to a_\infty > 0$ and  
$\beta(x)\to  b_\infty > 0$, as $|x| \to \infty$. 
Aim of this work is to describe some phenomena that can occur when the coefficients are ``competing''. 

\end{abstract}

\section{Introduction}

In this paper we consider superlinear elliptic problems of the type
\beq\label{P}\tag{$\mathcal{P}$} 
\left\{
        \begin{array}{ll}
 -\Delta u+\alpha(x)u= \beta (x)|u|^{p-1}u      \quad     \hbox{in }\R^N, 
\\
 u(x)>0\quad\hbox{in }  \R^N, \qquad u \in H^1(\R^N),         
        \end{array}
\right.
\eeq
 where $N \ge2$, $p>1,  \ p<2^* -1 ={N+2\over N-2},$ if $N\ge 3,$ and the coefficients $\alpha$ and $\beta$ are positive functions such that $\lim_ {|x| \to \infty}\alpha(x)=  a_\infty > 0$ and  $\lim_ {|x| \to \infty}\beta(x)= b_\infty > 0.$
 
The interest in studying \eqref{P}  comes from  its strong connections with Mathematical Physics and with problems in biology. The most known related question is probably the search of some solitary waves in nonlinear equations of the Klein-Gordon or Schr\"odinger type, but Euclidean scalar field equations appear also in several other contexts, like nonlinear optics, laser propagation, population dynamics, constructive field theory (see for instance \cite{BL1,Strauss1,S2}). It is worth also observing that  another strong motivation for the researchers attention is the challenging feature of \eqref{P}. Indeed, in spite of its variational nature, a lack of compactness, due to the invariance of $\R^N$ under the action of the noncompact group of translations, prevents a straight application of  the usual variational methods.

 Starting from the pioneering papers \cite{BL1,Strauss2}, several existence and multiplicity results have been stated as well as qualitative properties of the solutions to \eqref{P} have been studied.  The earliest results  were obtained in radially symmetric situations, taking advantage of the compact embedding in $L^P (\R^N),\ p\in (2, 2N/(N-2))$ of the subspace of $H^1(\R^N)$ consisting of radial functions. 
 On the contrary, when the coefficients do not enjoy symmetry, many different devices have been exploited to face the difficulties and to obtain the desired solutions. Describing all the various and interesting contributions in this direction, without forgetting something, is not an easy matter.  Thus, we prefer to focus the attention just on those  results more related to the subject of the present paper and refer readers, who are interested in a more detailed description of the research development,  to some survey papers \cite{C1,C2} and references therein. 

 When one considers the question of the existence of solutions of \eqref{P} in the non-symmetric case, first observation is that the topological situation and, then, the variational tools to be used are different  according the way in which $\alpha$ and $\beta$ approach their limit at infinity. When $\alpha(x)\rightarrow a_\infty$ from below and $\beta(x)\to b_\infty$ from above, as $|x| \to +\infty$, the existence of a positive ground state solution to \eqref{P} can be shown by  using a minimization method together concentration-compactness type arguments (see for example \cite{L,R}). Conversely, if $\alpha(x)\rightarrow a_\infty$ from above and  $\beta(x)\rightarrow b_\infty$ from below, \eqref{P} may not have a least energy solution. This is the case, for instance, when $\alpha(x) = a_\infty, \  \beta(x)\le b_\infty $ and  $\beta(x)\neq b_\infty $ on a positive measure set. Nevertheless, it is well known that also these situations can be successfully handled (see \cite{BL,BaLio}). Indeed,  adding to the previous conditions the  assumption on $b_\infty -\beta (x)$ of a suitable exponential ``fast" decay, it is possible to show  that a positive, not ground state, solution exists, by using minimax arguments together with delicate topological tools and a deep study of the nature of the obstacles to the compactness.

It is worth observing that most results concern cases, as those above described, in which the coefficients $\alpha$ and $\beta$ act on \eqref{P} in a ``cooperative'' way, even if in \cite{BaLio} some statements and comments including also different situations can be found.

 Purpose of this paper is  to describe some phenomena that can occur when the coefficients are ``competing''. In order to describe our work and results, let us first write \eqref{P}, the coefficients, and the assumptions on them, in a more appropriate way to our aim.  

We set
$$ \alpha(x) = a_\infty + \lambda a(x), \qquad  \beta(x) = b_\infty +  b(x)$$
 where $\lambda \in \R^+ $  and we assume

\begin{enumerate} [label=($H_\arabic{*}$), ref=$H_\arabic{*}$]
 \item \label{it:h1}$a\in L^{N/2}(\R^N)$, $a\ge 0$, $a\neq 0,$ and   $\lim_{|x|\to \infty}a(x)=0$;
 \item \label{it:h2}$b\in L^\infty(\R^N)$, $b\ge 0$, $b\neq 0,$ and   $\lim_{|x|\to \infty}b(x)=0$.
\end{enumerate}

\noindent Therefore, we rewrite \eqref{P}  as
\beq\label{eql}\tag{$\mathcal{P}_\l$}
\left\{
\begin{array}{ll}
-\Delta u +\big(a_\infty+\l a(x)\big)u=\big(b_\infty+ b(x)\big)|u|^{p-1}u, &\hbox{in }\R^N,
\\
u\in H^1(\R^N).
\end{array}
\right.
\eeq

As before reported, if $\l = 0,$ $(\mathcal{P}_\l)$ admits a positive ground state solution, corresponding to a solution of the minimization problem:
$$  m_\l:= \min \left\{ I_\l (u) : u \in H^1(\R^N)\setminus \left\{0\right\}, I'_\l(u) [u] = 0 \right\},$$
where $I_\l : H^1(\R^N)\to \R$ is the functional defined as
\beq\label{funzio}
I_\l(u)
=\frac {1}{2} \irn |\nabla u|^2+\big(a_\infty+\l a(x)\big)u^2
-\frac{1}{p+1}\irn\big(b_\infty+b(x)\big)|u|^{p+1}.
\eeq
Clearly, for  $\l$ suitably close to zero one expects that the situation does not change, but it is a natural question to wonder under which assumptions, when $\l$ increases, the minimum persists and, on the contrary, when it can be lost by the increasing competing effect of $a_\infty + \l a(x)$ against $b_\infty +  b(x).$ In this paper we present some contribution to this subject. 

Once introduced, in Section \ref{se:pre}, the variational framework and collected some useful facts and relations, we study, in Section \ref{se:ml}, the behaviour of the map $\l \rightarrow m_\l$. We show that it is a non decreasing, bounded, continuous map, and from its properties we deduce  that if a value $\bar{\l}$  exists for which the ground solution does not exist, then, \eqref{eql} does not admit a ground state solution also for  all $\l > \bar{\l}$. Therefore, the set of $\l$'s for which \eqref{eql} has not a ground state solution can be either empty set or an half-line. 

Then, in Section \ref{se:proofs}, we show that both the above considered cases can occur, according to the decay of $a(x)$ and $b(x),$ and, moreover, we prove that, 
also when the ground state solution does not exist, if $a(x)$ decays in a suitable fast way, the existence of a positive, not ground state, solution to \eqref{eql} can be obtained.  

The results in this direction can be summarized in the two following theorems. The first one concerns a situation in which the decay rate of $a(x)$ is faster than that  of $b(x),$ then, whatever $\l\in \R^+$ is, a ground state solution to \eqref{eql} exists.

\begin{theorem}
\label{th:<}
Suppose that \eqref{it:h1} and \eqref{it:h2} hold and that, 
\beq\label{ABforte1} \tag{$H_3$}
\irn a(x) e^{\a\sqrt{a_\infty}|x|}<+\infty \quad \hbox{ and } \quad
\lim_{|x|\to +\infty}b(x) e^{\b\sqrt{a_\infty}|x|}\ge c>0,
\eeq 
with $\b <\min\{2,\a\}$.
%
Then, for all $\l \in \R^+,$ \eqref{eql} admits a  ground state positive solution.	
\end{theorem}

On the other hand, the second statement apply to the opposite case, that is when the decay rate of  $a(x)$ is slower or equal to that  of $b(x).$

\begin{theorem}
\label{th:seconda}

Suppose that \eqref{it:h1} and \eqref{it:h2} hold and that 
there exists $\s \in (0,\sqrt{a_\infty})$ such that
\beq\label{ABforte2}\tag{$H_4$}
\lim_{|x|\to +\infty}a(x) e^{\a\sqrt{\s}|x|}\ge c>0 \quad \hbox{ and } \quad
\irn b(x) e^{\b\sqrt{\s}|x|}<+\infty,
\eeq 
with $\a\le\min\{p+1,\b\}$.
%
Then a number $\l^*>0 $ exists such that, for all $\l < \l^*,$ \eqref{eql} admits a ground state positive solution, while  if $\l > \l^*,$ \eqref{eql} has no ground state solution.

Furthermore, if the additional condition is satisfied
\beq\label{it:h3}\tag{$H_5$}
\irn a(x) |x| ^{N-1}e^{2\sqrt{a_\infty}|x|}<+\infty ,
\eeq
\eqref{eql} has, even for $\l \in [\l^*, +\infty)$, a positive solution. 
\end{theorem}

Lastly we observe that, if $a(x)$ and $b(x)$ enjoy radial symmetry, a multiplicity result can be obtained as well as some information about the nature of the radial solution (whose existence comes for all $\l$ just assuming \eqref{it:h1} and \eqref{it:h2}). 

\begin{theorem}
\label{th:rad2} Suppose that $a(x)$ and $b(x)$ are radially symmetric functions and that \eqref{it:h1} and \eqref{it:h2} hold. Assume moreover either  
\eqref{ABforte1} or \eqref{ABforte2}-\eqref{it:h3}. Then, there exists a number $\tilde{\l} >0$ such that, for all $\l > \tilde{\l}, $ 
 \eqref{eql} admits at least two positive solutions, and  the one that is radially symmetric  is not  ground state.

\end{theorem}

\section{Preliminaries}\label{se:pre}

In what follows, we will use the following notation:
\begin{itemize}
\item $\|u\|=\left(\irn |\n u|^2+a_\infty u^2\right)^{1/2}$ denotes the norm of $\H$;
\item $\|\cdot\|_{s}$ denotes the $L^s(\RN)$-norm;
\item $B_R(y)$ denotes the ball of radius $R$ centered at $y$;
\item $B_R$ denotes the ball of radius $R$ centered at $0$;
\item $c,c_i$ are positive constants which may vary from line to line;
\item $\Hr:=\{u\in \H : u \hbox{ is radially symmetric}\}$. 
\end{itemize}

\medskip

Solutions of \eqref{eql} are critical points of the functional $I_\l$ defined in \eqref{funzio}.

It is not difficult to verify that, whatever $\l \in \R^+$ is, the functional $I_\l$ is bounded neither from above nor from below. Hence, it is convenient
to consider $\Il$ restricted to a natural constraint, the Nehari manifold, that contains all the critical points of $\Il$ and on which  $\Il$ is bounded 
from below.

We set
\[
\Ne_\l:=\{u\in \H\setminus \{0\} :  \Il'(u)[u]=0\}.
\]
Next lemma contains the statement of the main properties of $\Ne_\l $.

\begin{lemma}\label{le:nehari}
Let \eqref{it:h1} and \eqref{it:h2} hold. Then for all  $\l \in \R^+$,
\begin{itemize}
\item[a)] $\Ne_\l$ is a $C^1$ regular manifold diffeomorphic to the sphere of $\H$;
\item[b)] $\Il$ is bounded from below on $\Ne_\l $ by a positive constant;
\item[c)] $u$ is a free critical point of $\Il$ if and only if $u$ is a critical point of $\Il$ constrained on $\Ne_\l $.
\end{itemize}
\end{lemma}

\begin{proof}
Let $\l \ge 0$ be fixed.
\\
{\it a)} Let $u\in \H \setminus \{0\}$ be such that $\|u\|=1$. Then there exists a unique $t\in (0,+\infty)$ for which $t u\in \Ne_\l$.
Actually, considering the equation
\[
0=\Il'(tu)[tu]
=t^2 \left[
\irn  |\n u|^2+(a_\infty +\l a(x))u^2 
-t^{p-1}\irn (b_\infty+b(x))|u|^{p+1}  
\right],
\]
it is clear that it admits a unique positive solution $t_\l(u)>0$ and that the corresponding point $t_\l(u)u\in \Ne_\l $, the ``projection'' of $u$ on $\Ne_\l $,
is such that 
\[
\Il(t_\l(u)u)=\max_{t\ge 0}\Il(tu).
\]
Now, let $u\in \Ne_\l$, using \eqref{it:h1}, \eqref{it:h2} and the continuity of the embedding of $\H$ in $L ^{p+1}(\RN)$, we deduce 
\beq\label{210}
c_1\|u\|^2_{p+1}
\le c_2 \|u\|^2
\le \irn |\n u|^2+(a_\infty +\l a(x))u^2
=\irn (b_\infty+b(x))|u|^{p+1} 
\le C\|u\|^{p+1}_{p+1},
\eeq
which implies the relations
\beq \label{211}
\|u\|_{p+1}\ge   c>0, \qquad \|u\|\ge  c>0.
\eeq
Moreover, setting $G_\l(u) :=\Il'(u)[u]$, $G_\l \in C^1(\H, \R)$ follows from the regularity of $\Il$ and, using \eqref{211}, we get 
\beq\label{212}
G_\l '(u)[u]=-(p-1)\irn |\n u|^2+(a_\infty +\l a(x))u^2 
\le -(p-1)  c<0.
\eeq
{\it b)} For all $u\in \Ne_\l $
\beq\label{213}
\Il(u)= \left(\frac 12 -\frac 1{p+1}\right)\irn |\n u|^2+(a_\infty +\l a(x))u^2 
\ge \left(\frac 12 -\frac 1{p+1}\right)\|u\|^2
\ge c>0.
\eeq
\\
{\it c)} If $u\neq0$ is a critical point of $\Il$, then $\Il'(u)=0$ and then $u\in \Ne_\l $. On the other hand, if $u$ is a critical point of $\Il$ constrained on $\Ne_\l $, 
then $\mu >0$ exists such that 
\[
0=\Il'(u)[u]
=G_\l (u)=\mu G_\l '(u)[u],
\]
from which, considering \eqref{212}, $\mu =0$ follows.
\end{proof}

We stress that the inequalities in \eqref{210} are not affected by $\l $, so relations \eqref{211} are true for all $\l\in \R^+ $ and $u\in \Ne_\l $ with $c$ independent of 
$\l $ and $u$. We state explicitly this fact in the following
\begin{corollary}\label{co:p+1}
Suppose that \eqref{it:h1} and \eqref{it:h2} hold. There exist constants $\bar c>0$ and $C>0$ such that for all $\l>0$ and for all $u\in \Ne_\l$, the relation
\[
\bar c\le \|u\|_{p+1}\le C\|u\|
\]
holds true.
\end{corollary}

In what follows we consider also the ``limit'' functional 
$I_\infty: \H \to \R$, defined as 
\[
I_\infty(u)
=\frac 12 \irn |\n u|^2+a_\infty u^2
-\frac{1}{p+1}\irn b_\infty |u|^{p+1}
\]
and the related natural constraint 
\[
\Ne_\infty:=\{u\in \H\setminus \{0\} :  I_\infty'(u)[u]=0\}.
\]
Critical points of $I_\infty$ are solutions of the ``limit problem at infinity''
\beq\label{eqi}\tag{$\mathcal{P}_\infty$}
\left\{
\begin{array}{ll}
-\Delta u +a_\infty u=b_\infty |u|^{p-1}u, &\hbox{in }\RN,
\\
u\in \H.
\end{array}
\right.
\eeq
Clearly, the conclusions of Lemma \ref{le:nehari} hold true for $I_\infty$ and $\Ne_\infty$, too, and, in what follows, for any given $u\in \H\setminus\{0\}$, 
we denote its projection on $\Ne_\infty$ by $\t(u)u\in \Ne_\infty$. It is well known \cite{BL1} that, setting
\[
m_\infty:=\inf\{I_\infty(u), u\in \Ne_\infty\},
\]
$m_\infty>0$ is achieved by a radially symmetric function $w$, unique up to translations \cite{K}, decreasing when the radial coordinate increases and such that
\beq\label{wdecay}
\lim_{|x|\to +\infty}|D^jw(x)||x|^\frac{N-1}{2}e^{\sqrt{a_\infty}|x|}=d_j>0, \quad j=0,1, \quad
 d_j\in \R.
\eeq
It is worth also observing that, for any changing sign critical point $u$ of $I_\infty$, the inequality
\beq\label{31}
I_\infty(u)\ge 2m_\infty
\eeq
holds. Indeed, if $u=u^+-u^- $ is a critical point of $I_\infty$ with $u^+\neq 0$ and $u^-\neq 0$, we have 
\[
\|u^\pm\|^2=\|u^\pm\|_{p+1}^{p+1}
\]
and so
\[
\left(\frac 12 -\frac 1{p+1}\right)\|u^\pm\|^2\ge m_\infty=\left(\frac 12 -\frac 1{p+1}\right)\|w\|^2,
\]
hence
\[
I_\infty(u)=\left(\frac 12 -\frac 1{p+1}\right)\left[\|u^+\|^2 +\|u^-\|^2\right]\ge 2m_\infty.
\]

In what follows, for any $y\in \RN$, we use the notation
\beq\label{wy}
w_y:=w(\cdot-y).
\eeq


\begin{proposition}\label{pr:<}
Suppose that \eqref{it:h1} and \eqref{it:h2} hold. Let $\l \ge 0$. Set
\beq\label{ml}
m_\l:=\inf\{I_\l(u), u\in \Ne_\l\}.
\eeq
Then 
\beq \label{221}
0<m_\l\le m_\infty.
\eeq
\end{proposition}

\begin{proof}
Let $\l \ge 0$ be fixed. First inequality in \eqref{221} is a straight consequence of \eqref{213}. To show that $m_\l \le m_\infty$, it is enough to build a 
sequence $(u_n)_n$, $u_n \in \Ne_\l $, such that $\lim_{n}\Il (u_n)\le m_\infty$. To this end, let us consider $(y_n)_n $, with $y_n \in \RN, |y_n|\to +\infty$, 
as $n\to +\infty$
and set $u_n=t_n w_{y_n}$, where $w_{y_n}$ is defined in \eqref{wy} and $t_n =t_\l (w_{y_n})$ is such that $u_n=t_n w_{y_n}\in \Ne_\l$. We have
\begin{align*}
I_\l(u_n)&=\frac{t_n^2}{2} \irn |\n w_{y_n}|^2+\big(a_\infty+\l a(x)\big)w_{y_n}^2 
-\frac{t_n^{p+1}}{p+1}\irn\big(b_\infty+b(x)\big)w_{y_n}^{p+1}
\\
&=\frac{t_n^2}{2} \left[ \|w\|^2+\l \irn a(x+y_n)w^2 \right]
-\frac{t_n^{p+1}}{p+1} \left[b_\infty\|w\|_{p+1}^{p+1}+\irn b(x+y_n)w^{p+1}\right].
\end{align*}
Moreover we have
\[
t_n^{p-1} 
=\frac{ \|w\|^2+\l \irn a(x+y_n)w^2 }{b_\infty\|w\|_{p+1}^{p+1}+\irn b(x+y_n)w^{p+1}}
\]
and, since it is quite clear that 
\[
\lim_n \irn a(x+y_n)w^{2}=0
\]
and 
\[
\lim_n \irn b(x+y_n)w^{p+1}=0,
\]
we deduce that,  as  $n\to +\infty$,
\beq  \label{22'}
 t_n \to 1,
\qquad
\Il (u_n) \to m_\infty. 
\eeq
\end{proof}

Next proposition states a result which is a straight application of the well known concentration-compactness principle \cite{L} and maximum principle.

\begin{proposition}\label{pr:22'}
If the strictly inequality 
\[
m_\l <m_\infty
\]
holds, $m_\l$ is achieved by a positive function and, moreover, all the minimizing sequences are relatively compact.
\end{proposition} 

As immediate consequence of Proposition \ref{pr:22'} we get
\begin{proposition}\label{pr:23}
Let \eqref{it:h1} and \eqref{it:h2} hold. Let $\l =0$. Then \eqref{eql} has a ground state positive solution.
\end{proposition}

\begin{proof}
The inequality 
\[
m_0 <m_\infty
\]
is easily obtained testing $I_0$ with $w$. Hence the claim follows using Proposition \ref{pr:22'}.
\end{proof}

We recall also a representation theorem of the Palais-Smale sequences (\cite{BC}, see also \cite{BaLio}) which is an useful tool when the equality
$m_\l =m_\infty$ occurs.

\begin{lemma}\label{le:comp}
Suppose that \eqref{it:h1} and \eqref{it:h2} hold. Let $(u_n)_n$ be a (PS) sequence of $\Il$ constrained on $\Ne_\l $, namely $u_n \in \Ne_\l $ and 
\begin{align*}
&\Il(u_n) \hbox{ is bounded};
\\
&\n {I_\l}_{\vert \Ne_\l} (u_n)\to 0 \hbox{ strongly in }\H.
\end{align*}
Then, up to a subsequence, there exist a solution $\bar u$ of \eqref{eql}, a number $k\in \N\cup\{0\}$, $k$ functions $u^1,\ldots,u^k$ of $\H$ and $k$ 
sequences of points $(y^j_n)_n,y^j_n\in \RN, 0\le j\le k$ such that, as $n \to +\infty$,
\begin{align*}
& u_n -\sum_{j=1}^{k}u^j(\cdot- y^j_n)\to \bar u, \hbox{ in }\H;
\\
&\Il(u_n)\to \Il(\bar u)+\sum_{j=1}^{k}I_\infty (u^j);
\\
&|y^i_n|\to +\infty,\ |y^j_n-y^i_n|\to +\infty \hbox{ if }i\neq j;
\\
& u^j \hbox{ are weak solutions of \eqref{eqi}}.
\end{align*}
Moreover we agree that in the case $k=0$, the above holds without $u^j$.
\end{lemma}

\begin{corollary}\label{co:24bis}
Assume that \eqref{it:h1} and \eqref{it:h2} hold. If $m_\l =m_\infty$, then the functional $\Il$ satisfies the (PS) condition at level $c$,  for all 
$c\in  (m_\infty, 2 m_\infty)$.
\end{corollary}

\begin{proof}
Let $(u_n)_n$ be a (PS) sequence of $\Il$ constrained on $\Ne_\l $ at level $c$, with $c\in  (m_\infty, 2 m_\infty)$, and apply Lemma \ref{le:comp}.
The claim follows recalling that any solution $u$ of \eqref{eqi} verifies $I_\infty(u)\ge m_\infty$ and, if it changes sign, $I_\infty(u)\ge 2 m_\infty$, 
furthermore any critical point $\bar u$ of $\Il$ is such that $\Il(\bar u)\ge m_\l =m_\infty$.
\end{proof}

\begin{lemma}\label{le:24ter}
Let \eqref{it:h1} and \eqref{it:h2} hold. Let $\l >0$. If $u\in \Ne_\l$ is a critical point of $\Il$ with $\Il(u)<2m_\l $, then $u$ does not
change sign.
\end{lemma}

\begin{proof}
The same argument used to show \eqref{31} allows to conclude that a changing sign solution $\bar u$ of \eqref{eql} must satisfy $\Il(\bar u)\ge 2 m_\l $.
\end{proof}

Next lemmas analyse the behavior of some sequences of functions $(u_n)_n$, $u_n \in \Ne_{\l_n}$, on which the energy is bounded.

\begin{lemma}\label{le:un<c1}
	Let  \eqref{it:h1} and \eqref{it:h2} hold. Let $(\l_n)_n$ be a sequence of positive numbers and, 
	for all $n\in \N$, let $u_n\in \Ne_{\l_n}$ be such that $I_{\l_n }(u_n)\le C$. 
	Then $(u_n)_n$ is bounded in $\H$. 
\end{lemma}

\begin{proof}
	Being
	\beq\label{numero}
	\Iln(u_n)
	=\left(\frac 12 - \frac1{p+1}\right)\left(\|u_n\|^2+\ln \irn a(x)u_n^2 \right)\le C,
	\eeq
	the claim easily follows.
\end{proof}

\begin{lemma}\label{le:un<c2}
Let  \eqref{it:h1} and \eqref{it:h2} hold. Let $(\l_n)_n$ be a diverging sequence of positive numbers and, 
for all $n\in \N$, let $u_n\in \Ne_{\l_n}$ be such that $u_n>0$ and $I_{\l_n }(u_n)\le C$. 
Then
the following relations hold true as $n\to +\infty$:
\begin{align}
& u_n\vert_{B_R}\to 0, \quad\hbox{ in }L^2(B_R), \ \hbox{ for all }R>0, \hbox{ as }n\to +\infty, \label{l2br}
\\
&\ln \irn a(x)u_n^2 \le C,\label{a}
\\
&\irn b(x)|u_n|^{p+1} \to 0, \quad \hbox{ as }n \to +\infty. \label{b}
\end{align}
Moreover, a positive constant $c>0$ and sequence $(y_n)_n$, with $y_n\in \RN$ and $|y_n|\to +\infty$, exist for which
\[ 
\lim_n \int_{B_1(y_n)}u_n^2\ge c.
\]
\end{lemma}

\begin{proof}
Since $(u_n)_n$ verifies \eqref{numero}, 
we obtain \eqref{a}. The divergence of $\l_n$ and \eqref{a} give \eqref{l2br}.
\\
Now, to show \eqref{b}, let us observe that, to any $\eps>0$, there corresponds an $R>0$ such that $0\le b(x)\le \eps$, for all $x\in \RN \setminus B_R$, hence we get
\[
\int_{\RN \setminus B_R} b(x)|u_n|^{p+1}\le C \eps.
\]
On the other hand, by interpolation and \eqref{l2br} we have
\[
\int_{B_R} b(x)|u_n|^{p+1}
\le C \left(\int_{B_R} u_n^2\right)^{\frac{\a(p+1)}2}\left(\int_{B_R} |u_n|^{2^*}\right)^{\frac{(1-\a)(p+1)}{2^*}} \xrightarrow[n\to +\infty]{} 0,
\]
where $\a \in (0,1)$, therefore, \eqref{b} follows.
\\
Lastly, 
\[
\lim_n \sup_{y\in \RN} \int_{B_1(y)}u_n^2=0
\]
cannot happen, because,  by \cite[Lemma I.1]{L}, it would imply that $u_n \to 0 $ in $L^{p+1}(\RN)$ contradicting \eqref{211}. 
Therefore there exist a positive constant $c$ and a sequence $(y_n)_n$, $y_n \in \RN$ such that 
\[
\int_{B_1(y_n)}u_n^2\ge c>0,
\]
and considering \eqref{l2br}, we conclude that $|y_n|\to +\infty$.
\end{proof}

We end this section with two lemmas dealing with asymptotic estimates: the first one is well known and concerns the exponential decay of elliptic problem solutions, the proof can be essentially found in \cite{S} and we present it for sake of completeness;  the second one can be proved arguing as in the proof of Proposition 1.2 of~\cite{BaLio}.

\begin{lemma}\label{le:exp}
Let $u\in \H$ be a solution of 
\[
\left\{
\begin{array}{ll}
	-\Delta u +\a (x)u=\b(x)|u|^{p-1}u, &\hbox{in }\RN,
	\\
	u\in \H,
\end{array}
\right.
\]
where $\a (x)\to a_\infty$ and $\b(x)\to b_\infty$, as $|x|\to +\infty$. Then, for any $0<\s<a_\infty$,  there exists $C>0$ such that
\beq\label{2101}
|u(x)|\le C e^{-\sqrt{\s}|x|}, \qquad\hbox{for all }x\in \RN.
\eeq
\end{lemma}

\begin{proof}
The proof relies on some ideas of \cite{CM10,S}.
\\
Arguing as in \cite[Proof of Theorem 1.4]{CM10}, we first study by a bootstrapping procedure the regularity of $u$ and we show that
\beq\label{eq:step1}
u(x)\to 0, \quad \hbox{ as }|x|\to +\infty.
\eeq
Then, we prove \eqref{2101} for $u^+$, the positive part of $u$. The arguments for $u^-$ are similar.
\\
Let $\s $ belong to $(0,a_\infty)$. The function $u^+$ solves
\beq \label{67}
\left\{
\begin{array}{ll}
-\Delta u +\s u =\b(x)u^p-\big(\a(x)-\s \big)u& \hbox{in }\O^+,
\\
u\in H^1_0(\O^+),
\end{array}
\right.
\eeq
where $\O^+=\{x\in \RN :   u(x)>0\}$. If $\O^+$ is bounded, the claim is trivial. So we 
suppose $\O$ unbounded. By \eqref{eq:step1}, there exists a number $R > 0$ such that
\[
\big(\a(x)-\s \big)u^+ -\b(x)(u^+)^p>0,\quad \forall x \in \O^+\cap B_R^c.
\]
Let us denote by $\g$ the fundamental radial solution of
\beq \label{68}
\left\{
\begin{array}{ll}
-\Delta v +\s v =0& \hbox{in }B_R^c,
\\
v(R)=\max_{|x|=R}u^+.
\end{array}
\right.
\eeq
It is well known (see e.g. \cite{BS}) that
\beq \label{69}
\g (r)|r|^{\frac{N-1}{2}}e^{\sqrt{\s}r}\to c>0, \quad \hbox{ as }|r|\to +\infty.
\eeq 
Observe that (\ref{67}) and (\ref{68}) imply that $\eta:= \g - u^+$ solves
\[
-\Delta \eta+ \s \eta=\big(\a(x)-\s)u^+ -\b(x)(u^+)^p>0,\quad \hbox{ in }\O^+\cap B_R^c.
\]
Then, by the weak maximum principle,
\[
\inf_{\O^+\cap B_R^c} \eta\ge \inf_{\de(\O^+\cap B_R^c)}\eta\ge 0,
\]
hence, 
\[
0\le u^+(x)\le \g (x) \quad\hbox{ for all }x\in  B_R^c.
\]
This, together with \eqref{69}, proves the lemma.
\end{proof}

\begin{lemma}\label{riccardo}
If $g\in L^\infty(\RN)$ and $h\in L^1(\RN)$ are such that, for some $\a \ge0$, $b\ge 0$, $\g \in \R$
\[
\lim_{|x|\to +\infty}g(x)e^{\a |x|}|x|^b=\g 
\quad
\hbox{ and } \quad
\irn |h(x)|e^{\a |x|}|x|^b<+\infty,
\]
then, for every $z\in \RN \setminus\{0\}$,
\[
\lim_{\rho\to +\infty} \left(\irn g(x+\rho z)h(x)\right) e^{\a |\rho z|}|\rho z|^b
=\g \irn h(x)e^{-\a (x \cdot z)/|z|}.
\]
\end{lemma}

\section{Properties of the map $\l \mapsto m_\l $}\label{se:ml}

We start this section showing a monotonicity property of the map $\l \mapsto m_\l $.

\begin{proposition}\label{pr:mono}
Let  \eqref{it:h1} and \eqref{it:h2} hold. The map $\l\in \R^+ \mapsto m_\l $ is monotone non-decreasing.
\end{proposition} 

\begin{proof}
Let $u\in \H\setminus \{0\}$, $\l \in \R^+$ and $t_\l(u)$ be such that $t_\l(u) u\in \Ne_\l $, namely
\[
[t_\l(u)]^{p-1}=\frac{\|u\|^2+\l \irn a(x)u^2}{b_\infty\|u\|^{p+1}_{p+1}+\irn b(x)|u|^{p+1}}.
\] 
Clearly, if $\l_1, \l_2\in \R^+$ are such that $\l_1<\l_2$, then $t_{\l_1}(u)\le t_{\l_2}(u)$,
so
\begin{align*}
I_{\l_1}(t_{\l_1} u)
&=\left(\frac12 -\frac1{p+1}\right)[t_{\l_1}(u)]^{2}\left[\|u\|^2+\l_1 \irn a(x)u^2\right]
\\
&\le \left(\frac12 -\frac1{p+1}\right)[t_{\l_2}(u)]^{2}\left[\|u\|^2+\l_2 \irn a(x)u^2\right]
=I_{\l_2}(t_{\l_2}u).
\end{align*}
Therefore, by the arbitrariness of $u$, we conclude that $m_{\l_1}\le m_{\l_2}$.
\end{proof}

\begin{remark}\label{re:tl}
Let $u\in \H\setminus \{0\}$ and $\l_1, \l_2\in \R^+$ be such that $\l_1<\l_2$. Then 
\[
t_{\l_1}(u)= t_{\l_2}(u) \ \hbox{ if and only if }\ \irn a(x)u^2=0.
\]
\end{remark}

\begin{proposition}\label{pr:mbar}
Let  \eqref{it:h1} and \eqref{it:h2} hold.
If $\bar \l \in \R^+$ is such that $m_{\bar \l}=m_\infty$, then, for all $\l >\bar \l $, $m_\l =m_\infty$ and, moreover, it is not achieved.
\end{proposition}

\begin{proof}
Let $\l >\bar \l$ be fixed. 
\\
The equality $m_\l =m_\infty$ is a direct consequence of Propositions \ref{pr:<} and \ref{pr:mono}: indeed we have
\[
m_\infty=m_{\bar \l}\le m_{\l}\le m_\infty.
\]
Let us now prove that $m_\l $ is not achieved. 
\\
Arguing by contradiction, we suppose that $u_\l \in \Ne_\l $ exists such that $\Il(u_\l)=m_\l =m_\infty$. 
Furthermore we can assume that $u_\l >0$, otherwise we can replace it by $|u_\l|$, because $|u_\l|\in \Ne_{\l}$ and 
$\Il(|u_\l|)=\Il(u_\l)=m_\l =m_\infty$, then the maximum principle implies that $|u_\l|>0$.
\\
Let $t_{\bar \l}=t_{\bar \l}(u_\l)>0$ be such that $t_{\bar \l}u_\l\in \Ne_{\bar \l}$. Arguing as in Proposition \ref{pr:mono} and 
considering Remark \ref{re:tl}, we get $t_{\bar \l}< 1$, thus
\[
m_\infty=m_{\bar \l}\le I_{\bar \l}(t_{\bar \l} u_\l)< I_{\l}( u_\l)=m_\infty,
\]
reaching a contradiction.
\end{proof}


An immediate consequence of Proposition \ref{pr:mbar} is the following
\begin{corollary}
Let  \eqref{it:h1} and \eqref{it:h2} hold. There exists at most a unique $\bar \l\in \R^+$ such that $m_{\bar \l}=m_\infty$ and it is achieved.
\end{corollary}

Let us define 
\beq \label{l*}
\l ^*:=\sup\{\l \in \R^+:m_\l <m_\infty\}.
\eeq 

\begin{proposition}\label{pr:l*}
Suppose that $\l ^*<+\infty$, then  
\beq \label{F}
m_{\l^*} =m_\infty
\eeq 
and the following equality holds
\beq\label{351}
\sup\{\l \in \R^+:m_\l <m_\infty\}=\min\{\l \in \R^+:m_\l =m_\infty\}.
\eeq 
\end{proposition}

\begin{proof}
Assume by contradiction that $m_{\l^*} <m_\infty$. Then there exists, therefore, $u^*\in \Ne_{\l ^*}$ such that $I_{\l ^*}(u_{\l ^*})=m_{\l ^*}$.
Let $(\ln)_n$ be a sequence of numbers such that $\ln \searrow \l ^*$. By definition of $\l ^*$ \eqref{l*}, for all $n\in \N$, $m_{\ln}=m_\infty$.  
Moreover, denoting, for any $n\in \N$, $t_n:=t_{\ln}(u_{\l ^*})$ the number such that $t_n u_{\l ^*}\in \Ne_{\ln}$, 
in view of definition of $t_n$, we infer that $t_n\to 1$, as $n\to +\infty$. Therefore we get
\begin{multline*}
m_\infty=m_{\ln}
\le \Iln(t_n u_{\l ^*})
=\left(\frac 12 -\frac1{p+1}\right)t_n^2\left[\|u_{\l ^*}\|^2+\ln\irn a(x)u_{\l ^*}^2\right]
\\
\xrightarrow[n\to +\infty]{}\left(\frac 12 -\frac1{p+1}\right)\left[\|u_{\l ^*}\|^2+\l^*\irn a(x)u_{\l ^*}^2\right]
=I_{\l ^*}(u_{\l ^*})=m_{\l ^*}<m_\infty,
\end{multline*}
reaching a contradiction. Finally \eqref{351} is a consequence of \eqref{F} and Proposition \ref{pr:mbar}.
\end{proof}

The following proposition shows the continuity of the map $\l \in \R^+ \mapsto m_\l $.
\begin{proposition}\label{pr:cont}
Let  \eqref{it:h1} and \eqref{it:h2} hold. 
Then the map $\l \in \R^+ \mapsto m_\l $ is continuous.
\end{proposition}

\begin{proof}
We divide the proof in several steps, analyzing all the possible cases.  
\medskip
\\
{\em Case 1: $\l ^*=+\infty$}. 
\\
Let $\l \in \R^+$. Being $m_{\l}<m_\infty$, by Proposition \ref{pr:22'}, $u_{\l}\in \Ne_{\l}$
exists such that $I_{\l}(u_{\l})=m_{\l}$. 
Let $(\ln)_n$ be such that $\ln\to  \l $. For all $n$, let $t_n:=t_{\ln}(u_{\l})$ be such that $t_n u_{\l}\in \Ne_{\ln}$
then, using the definition of $t_n$, we get that $t_n\to 1$, as $n\to +\infty$, hence
\begin{multline*}
m_{\ln}
\le \Iln(t_n u_{ \l})
=\left(\frac 12 -\frac1{p+1}\right)t_n^2\left[\|u_{ \l}\|^2+\ln\irn a(x)u_{ \l}^2\right]
\\
\xrightarrow[n\to +\infty]{} \left(\frac 12 -\frac1{p+1}\right)\left[\|u_{ \l}\|^2+ \l\irn a(x)u_{ \l}^2\right]
=I_{ \l}(u_{ \l})=m_{ \l}.
\end{multline*}
Therefore we obtain
\beq \label{limsup}
\limsup_n m_{\ln}\le m_{ \l}.
\eeq
Since, for all $n\in \N$, $m_{\ln}<m_\infty$, $u_n\in \Ne_{\ln}$ exists by Proposition \ref{pr:22'} such that $\Iln(u_n)=m_{\ln}$ and by 
Lemma \ref{le:un<c1}, we can assert that the sequence $(u_n)_n$ is bounded in $\H$.
Let $\widehat{ t}_n:=t_{ \l}(u_n)$ be such that $\widehat{ t}_n u_n\in \Ne_{ \l}$. Being
\begin{align*}
1&= \frac{\|u_n\|^2+\ln \irn a(x)u_n^2}{b_\infty\|u_n\|^{p+1}_{p+1} +\irn b(x)|u_n|^{p+1}},
\\
(\widehat{ t}_n)^{p-1}&= \frac{\|u_n\|^2+ \l \irn a(x)u_n^2}{b_\infty\|u_n\|^{p+1}_{p+1} +\irn b(x)|u_n|^{p+1}},
\end{align*}
we deduce $ \widehat t_n\to 1$ and $|I_{ \l}( \widehat t_n u_n)- m_{\ln}|\to 0$, as $n\to +\infty$.
Thus we get
\[
m_{ \l}\le \liminf_n m_{\ln},
\]
that, together with \eqref{limsup}, brings to the conclusion.
\medskip
\\
{\em Case 2: $\l ^*\in \R^+$.}
\\
For any $\l <\l ^*$, we can argue as in Case 1. 
On the other hand, $\l  \mapsto m_\l $ is a constant map  in $(\l ^*,+\infty)$.  
Therefore, we need only to prove the continuity at $\l =\l ^*$.
\\
Let $(\ln)_n$ be a sequence of numbers such that $\ln\to \l ^*$. By \eqref{F}, if $\ln \searrow\l ^*$, 
the conclusion is trivial. Assume, therefore, that $\ln \nearrow\l ^*$. 
By \eqref{ml}, fixing arbitrarily $\e>0$, $u_\e\in \Ne_{\l ^*}$ can be found such that $I_{\l^*}(u_{\e})<m_{\l^*}+\e$. 
Then, denoting, for all $n$, by $t_{n,\e}:=t_{\ln}(u_{\e})$ the number such that $t_{n,\e} u_{\e}\in \Ne_{\ln}$,
we infer that $t_{n,\e}\to 1$, as $n\to +\infty$, and moreover
\begin{multline*}
m_{\ln}
\le \Iln(t_{n,\e} u_{\e})
=\left(\frac 12 -\frac1{p+1}\right)t_{n,\e}^2\left[\|u_{\e}\|^2+\ln\irn a(x)u_{\e}^2\right]
\\
\xrightarrow[n\to +\infty]{} \left(\frac 12 -\frac1{p+1}\right)\left[\|u_{\e}\|^2+\l^*\irn a(x)u_{\e}^2\right]
=I_{\l^*}(u_{\e})<m_{\l^*}+\e,
\end{multline*}
that gives
\[
\limsup_n m_{\ln}\le m_{\l^*}+\e.
\]
Hence, by the arbitrariness of $\e$, we obtain
\[
\limsup_n m_{\ln}\le m_{\l^*}.
\]
On the other hand, being $m_{\ln}<m_\infty$, for all $n$, we can argue as in Case 1 showing that 
\[
m_{\l^*}\le \liminf_n m_{\ln}.
\]
This conclude the proof.
\end{proof}

\begin{remark}
We observe that $\l ^*>0$, by the continuity of the map $\l\mapsto m_\l $ and the fact that $m_0<m_\infty$.
\end{remark}

\section{Proof of the results}\label{se:proofs}

\subsection{The case in which $a$ decays faster than $b$}

\

\begin{proposition}\label{pr:<2}
Let  \eqref{it:h1}, \eqref{it:h2} and \eqref{ABforte1} hold and let $\l ^*$ be the number defined in \eqref{l*}.
Then $\l ^*=+\infty$.
\end{proposition}

\begin{proof}
	We first observe that, by Proposition \ref{pr:23}, we already know that $m_0<m_\infty$, so in what follows we can assume $\l >0$.
	In order to show that, for all $\l >0$, $m_\l <m_\infty$, we fix arbitrarily a $\l >0$ and we test $\Il$ by functions $u_n\in  \Ne_\l$, $u_n=t_n w_{y_n}$ where $y_n$ and $t_n$ are as in the proof of Proposition \ref{pr:<}.
	By \eqref{22'}, we can assert that, up to a subsequence, $t_n\ge \bar c >0$.
	Now, we have
	\begin{align*}
	m_\l &\le I_\l(u_n)
	=I_\infty(t_{n}w_{y_n})
	+t_{n}^2 \irn \left(\frac{\l}2 a(x+y_n)w^2-\frac{t_{n}^{p-1}}{p+1}b(x+y_n)w^{p+1}\right)
	\\
	&\le I_\infty(w_{y_n})
	+t_{n}^2 \irn \left(\frac{\l}2 a(x+y_n)w^2-c\ b(x+y_n)w^{p+1}\right)
	\\
	&= m_\infty
	+t_{n}^2 \irn \left(\frac{\l}2 a(x+y_n)w^2-c\ b(x+y_n)w^{p+1}\right).
	\end{align*}
	Thus, we get the conclusion if we show that, for large $n$,
	\beq\label{int<0}
	\irn \left(\frac{\l}2 a(x+y_n)w^2-c\ b(x+y_n)w^{p+1}\right)<0.
	\eeq
First let us observe that, by \eqref{ABforte1},
\beq\label{intb}
\irn b(x+y_n)w^{p+1}
\ge  \int_{B_1} b(x+y_n)w^{p+1}
\ge \inf_{x\in B_1} b(x+y_n) \int_{B_1} w^{p+1}
\ge c e^{-\b\sqrt{a_\infty}|y_n|}.
\eeq
Then let us show that 
\beq \label{min2a}
	\irn a(x+y_n)w^2\le c_1 e^{-\min\{2,\a \}\sqrt{a_\infty}|y_n|}. 
\eeq
Indeed, if $\a \le 2$, by the exponential decay \eqref{wdecay} of $w$, we have $w^2(x)\le C e^{-\a \sqrt{a_\infty}|x|}$, for all $x\in \RN$, and moreover, by \eqref{ABforte1},  $a\in L^1(\RN)$. So, by Lemma \ref{riccardo}, we have
\[
\irn  a(x+y_n)w^2 \le C \irn  a(x+y_n)e^{-\a \sqrt{a_\infty}|x|}\le C e^{-\a \sqrt{a_\infty}|y_n|}.
\]
On the other hand, when $2<\a $, thanks to assumption \eqref{ABforte1}
\[
\irn a(x)e^{2\sqrt{a_\infty}|x|}<+\infty,
\]
and we can again apply Lemma \ref{riccardo}, obtaining
\[
\irn  a(x+y_n)w^2 \le C\irn  a(x+y_n)e^{-2 \sqrt{a_\infty}|x|}\le C e^{-2 \sqrt{a_\infty}|y_n|}.
\]
Thus \eqref{min2a} holds true.
\\
%
%
Since $\b <\min\{2,\a \}$, \eqref{int<0} follows by \eqref{intb} and \eqref{min2a}.
\end{proof}

\begin{proof}[Proof of Theorem \ref{th:<}]
It is an immediate consequence of Proposition \ref{pr:22'} and Proposition \ref{pr:<2}.
\end{proof}

\subsection{The case in which $a$ decays slower or equal to $b$}

\

To prove Theorem \ref{th:seconda}, we need more work. We start by a proposition that is basic to obtain the first part of the claim.

\begin{proposition}\label{pr:=}
Let \eqref{it:h1}, \eqref{it:h2}  and \eqref{ABforte2} hold and let $\l ^*$ be the number defined in \eqref{l*}.
Then $\l ^*\in \R^+$.
\end{proposition}

\begin{proof}
Suppose by contradiction that $\l ^*=+\infty$. Then, by Proposition \ref{pr:l*}, $m_{\l}<m_\infty$, for all $\l \in \R^+$.
Let $(\l_n)_n$ be a diverging sequence. By Proposition \ref{pr:22'} a sequence $(u_n)_n$ exists such that, for all $n\in \N$
\[ 
u_n > 0, \quad u_n\in \Ne_{\l_n},\quad I_{\l_n}(u_n)=m_{\l_n}<m_\infty \quad\hbox{and}\quad I_{\l_n}'(u_n)=0.
\]
By Proposition \ref{le:un<c1}, this sequence is bounded in $\H$.
Set $\t_n:=\t(u_n)$, i.e. $\t_n u_n\in \Ne_\infty$.
Let us observe first that
\begin{equation}\label{eq:iun}
I_{\l_n}(\t_n u_n)<I_{\infty}(\t_n u_n),
\end{equation} 
otherwise we would have
\[
m_\infty\le I_\infty(\t_n u_n)\le I_{\l_n}(\t_n u_n)\le  I_{\l_n}(u_n)=m_{\l_n}<m_\infty,
\]
and this is impossible.
\\
Moreover, by definition 
\[
\t_n^{p-1}=\frac{\|u_n\|^2}{\|u_n\|_{p+1}^{p+1}}, 
\]
so, by using Corollary \ref{co:p+1} and the boundedness of  $(\|u_n\|)_n$, we deduce that $\bar c, \bar C>0$ exist such that 
\beq \label{cC}
\bar c\le \t_n \le \bar C.
\eeq
Now, being \eqref{eq:iun} equivalent to 
\beq\label{eq:iuntn}
 \frac{\l_n}{2}\irn a(x)u_n^2
 -\frac{\t_n^{p-1}}{p+1}\irn b(x) |u_n|^{p+1}<0.
\eeq
in view of \eqref{b} and \eqref{cC}, we deduce
\beq\label{numero2}
\ln \irn a(x)u_n^2 \to 0, \quad \hbox{ as }n\to +\infty.
\eeq
Therefore, considering that $u_n \in \Ne_{\ln}$ and $\t_n u_n\in \Ne_\infty$, and using \eqref{b}, we obtain
\[
b_\infty (\t_n^{p-1}-1)\|u_n\|^{p+1}_{p+1}= \irn b(x)|u_n|^{p+1}-\ln \irn a(x)u_n^2=o_n(1),
\]
that, by Corollary \ref{co:p+1}, implies
\beq\label{tn1}
\lim_n \t_n =1.
\eeq
Hence, again using \eqref{b} and \eqref{numero2}, we have
\begin{align*}
m_\infty 
&> \Iln(u_n)\ge \Iln(\t_n u_n) 
\\
&=I_\infty(\t_n u_n)
+\frac{\t_n^2 \ln}2 \irn a(x)u_n^2
-\frac{\t_n^{p+1}}{p+1}\irn b(x)|u_n|^{p+1}
\\
&=I_\infty(\t_n u_n)+o_n(1)\ge m_\infty +o_n(1),
\end{align*}
that gives
\[
I_\infty(\t_n u_n) \to m_\infty, \quad\hbox{ as }n \to +\infty.
\]
By the uniqueness of the family of minimizers of $I_\infty$ on $\Ne_\infty$, then $(y_n)_n$ exists such that $y_n\in \RN$ and
\[
\t_n u_n -w_{y_n}\to 0\quad\hbox{ in }\H, \quad\hbox{ as }n\to +\infty.
\]
Thus, setting $v_n=u_n(\cdot+y_n)$, thanks to \eqref{tn1}, we deduce 
\[
v_n \to w\quad\hbox{ in }\H, \quad\hbox{ as }n\to +\infty.
\]
Since, for any $n$, $v_n$ is a solution of 
\[
-\Delta u+\big(a_\infty+\ln a(x+y_n)\big)u=\big(b_\infty+b(x+y_n)\big)u^{p+1}
\]
by virtue of the Schauder interior estimates (see e.g. \cite{serrin}), $v_n\to w$ locally in $C^2$ sense
and Lemma \ref{le:exp} applies to $v_n$.
\\
Now, to conclude, it is enough to prove the following

\medskip
\noindent
{\sc Claim:} for large $n$, the inequality
\beq\label{ivn}
 \frac{\l_n}{2}\irn a(x+y_n)v_n^2
 -\frac{3}{2(p+1)}\irn b(x+y_n) |v_n|^{p+1}>0
\eeq
holds true.
\\
Indeed, considering \eqref{tn1} and the definition of $v_n$, it is clear that, for large $n$, \eqref{ivn} contradicts \eqref{eq:iuntn}.
\\
By \eqref{ABforte2} and since $v_n \to w$ locally in $C^2$ sense, we have
\beq\label{av}
\irn a(x+y_n)v_n^2
\ge  \int_{B_1} a(x+y_n)v_n^2
\ge \inf_{x\in B_1} a(x+y_n) \int_{B_1} v_n^2
\ge c e^{-\a\sqrt{\s}|y_n|},
\eeq
while, by the exponential decay \eqref{2101} of $v_n$, arguing as in the proof of \eqref{min2a}, we have
\beq \label{minbp1}
	\irn b(x+y_n)|v_n|^{p+1}\le c e^{-\min\{p+1,\b \}\sqrt{\s}|y_n|}. 
\eeq
Therefore, since $\a \le \min\{p+1,\b \}$ and $(\ln)_n$ is a diverging sequence, \eqref{ivn} follows by \eqref{av} and \eqref{minbp1}.
\end{proof}

Now, we turn to build tools and topological variational techniques useful to prove the existence of an higher energy solution when \eqref{eql} has 
no ground state solutions.

First step is reminding  the definition of barycenter $\b $ 
of a function $u\in \H$, $u \neq 0$, given in \cite{CP03}. Setting
\begin{align*}
\mu(u)(x) &=
\frac 1 {| B_1 ( 0 )|} \int_{ B_1 ( x )}| u ( y ) | dy ,
\qquad
 \mu( u ) \in L^\infty(\RN)  \hbox{ and is continuous},
 \\
 \hat u ( x ) &= \left[
 \mu ( u )( x )- \frac 1 2 \max \mu ( u )( x )
 \right]^+, \qquad \hat u \in C_0(\RN);
 \end{align*}
 we define $\b : \H \setminus \{0\}\to \RN$ as
 \[
\b( u ) = \frac 1{|\hat u |_1} \irn x \hat u ( x ) dx \in \RN.
\]
Since $\hat u$ has compact support, $\b $ is well defined. Moreover the following properties hold:
\begin{enumerate}
\item $\b $ is continuous in $\H\setminus\{0\}$;
\item if $u$ is a radial function, $\b( u ) = 0$;
\item for all $t\neq 0$ and for all $u\in \H\setminus\{0\}$, $\b( tu ) = \b( u )$;
\item given $z\in \RN$ and setting $u_z ( x ) = u ( x -z )$ , $\b( u_z ) = \b( u )+ z$.
\end{enumerate}

Let 
\[
\mathcal{B}^\l_0
:=\inf\{\Il(u): u\in \Ne_\l, \b (u)=0\}.
\]

\begin{lemma}\label{le:mB}
Let \eqref{it:h1} and \eqref{it:h2} hold, let $\l \ge 0$ be fixed and moreover let $m_\l =m_\infty$ be not achieved. Then
\[
m_\l =m_\infty<\mathcal{B}^\l_0.
\]
\end{lemma}

\begin{proof}
Assume, by contradiction, that a sequence $(u_n)_n$, $u_n\in \Ne_\l$ exists such that $\b (u_n)=0$ and $\Il(u_n)= m_\infty +o_n(1)$. 
By the Ekeland variational principle, we can assert the existence of a sequence of functions $(v_n)_n$ such that
\beq\label{*} 
v_n\in \Ne_\l,\ \Il(v_n)=m_\infty+o_n(1),\ \Il'(v_n)=o_n(1)\ \hbox{ and } \ |\b(v_n)-\b(u_n)|=o_n(1).
\eeq
Since $m_\l $ is not achieved, $(v_n)_n$ cannot be relatively compact and, by Lemma \ref{le:comp}, the equality
\[
u_n = w_{y_n}+o(1),
\]
must be true with $|y_n|\to +\infty$, contradicting \eqref{*}.
\end{proof}

Let $\xi \in \RN$ with $|\xi|=1$ and $\Sigma=\de B_2(\xi)$. We set 
\beq\label{bwdef}
\bw=\frac{w}{\|w\|_{p+1}},
\eeq
and for any $y\in \RN$, $\bw_y=\bw(\cdot -y)$. Observe that $\bw$ satisfies
\beq\label{bw}
-\Delta \bw +a_\infty \bw ={\bf M}\bw^{p},
\eeq
where, by a direct computation, one can verify that
\beq\label{bfm} 
{\bf M}=b_\infty^\frac2{p+1}\left(\frac{2(p+1)}{p-1}m_\infty\right)^\frac{p-1}{p+1}.
\eeq 

For any $\rho>0$ and $(z,s)\in \Sigma\times[0,1]$, let 
\[
\psi_\rho(z,s)=(1-s)\bw_{\rho z}+s \bw_{\rho \xi}.
\]
Moreover let $\Psi_\rho:\Sigma\times[0,1]\to \Ne_\l $ be so defined
\[
\Psi_\rho(z,s)= t_{z,s}^\l\psi_\rho(z,s),
\]
where $t_{z,s}^\l>0$ is such that $t_{z,s}^\l\psi_\rho(z,s)\in \Ne_\l $.

\begin{lemma}\label{le:B<=max}
Suppose that \eqref{it:h1} and \eqref{it:h2} hold and let $\l >0$ be fixed. Then for all $\rho>0$, we have
\[
\mathcal{B}^\l_0\le \mathcal{T}_\rho^\l :=\max_{\Sigma\times [0,1]}\Il(\Psi_\rho(z,s)).
\]
\end{lemma}

\begin{proof}
Since $\b (\Psi_\rho(z,0))=\rho z$, we infer that $\b \circ\Psi_\rho(\Sigma\times \{0\})$ is homotopically equivalent in $\RN\setminus\{0\}$ to $\rho \Sigma$, 
then there exists $(\bar z, \bar s)\in \Sigma\times [0,1]$ such that $\b (\Psi_\rho (\bar z, \bar s))=0$ and, as consequence, 
\[
\mathcal{B}^\l_0
\le \Il (\Psi_\rho (\bar z, \bar s))
\le \mathcal{T}_\rho^\l .
\]
\end{proof}

\begin{lemma}\label{le:M2m}
Let the assumptions of Lemma \ref{le:B<=max} be true and suppose that \eqref{it:h3} holds. 
Then there exists $\rho_0$ such that, for all $\rho>\rho_0$,  
\[
\mathcal{T}_\rho^\l =\max_{\Sigma\times [0,1]}\Il(\Psi_\rho(z,s))<2m_\infty.
\]
\end{lemma}

\begin{proof}
The argument is quite similar to that used in \cite{CM03,CP95} so we only sketch the proof for sake of completeness and reader's convenience.
\\
Observe that
\[
\Il(\Psi_\rho(z,s))=\dfrac{p-1}{2(p+1)}
\left[
\dfrac{\dis \|\psi_\rho(z,s)\|^2+\l \irn a(x)\psi^2_\rho(z,s)}
{\dis \left( \irn (b_\infty+b(x))|\psi_\rho(z,s)|^{p+1}\right)^{\frac{2}{p+1}}}\right]^{\frac{p+1}{p-1}}.
\]
Let us evaluate
\begin{align*}
N_\rho^\l (z,s)
&:=\|\psi_\rho(z,s)\|^2+\l \irn a(x)\psi^2_\rho(z,s)
\\
&\ =(1-s)^2\|\bw_{\rho z}\|^2
+2s(1-s)(\bw_{\rho z},\bw_{\rho \xi})_{H^1}
+s^2\|\bw_{\rho \xi}\|^2
\\
&\quad+\l \left[(1-s)^2\irn a(x)\bw_{\rho z}^2
+2s(1-s)  \irn a(x)\bw_{\rho z}\bw_{\rho \xi}
+s^2 \irn a(x)\bw_{\rho \xi}^2\right].
\end{align*}
Since $\bw$ satisfies \eqref{bw}, $\|\bw_{\rho z}\|^2=\|\bw_{\rho \xi}\|^2={\bf M}$ and 
\[
(\bw_{\rho z},\bw_{\rho \xi})_{H^1}={\bf M}\irn \bw_{\rho z}^p\bw_{\rho \xi}
={\bf M}\irn \bw_{\rho z}\bw_{\rho \xi}^p.
\]
Therefore, by \cite[Proposition 1.2]{BL} and Proposition \ref{riccardo}  (see also \cite [Lemma 3.7]{ACR}), using \eqref{it:h3} and the fact that $|z|\ge 1$, we get
\begin{align*}
\e_{\rho}=\irn \bw_{\rho z}^p\bw_{\rho \xi}
&=\irn \bw_{\rho z}\bw_{\rho \xi}^p
\sim  |2\rho|^{-\frac{N-1}2}e^{-2\rho \sqrt{a_\infty}},
\\
\irn a(x)\bw_{\rho z}^2&
= o(\e_{\rho}),
\\
\irn a(x)\bw_{\rho \xi}^2&
= o(\e_{\rho}),
\end{align*}
and
\[
\irn a(x)\bw_{\rho z}\bw_{\rho \xi}\le c \left(\irn a(x)\bw_{\rho z}^2+\irn a(x)\bw_{\rho \xi}^2\right) = o(\e_{\rho}).
\]
Thus
\[
N_\rho^\l (z,s)=[(1-s)^2+s^2]{\bf M}
+2s(1-s){\bf M}\e_{\rho}+o(\e_{\rho}).
\]
Moreover, by \cite[Lemma 2.7]{CP95}, we have
\begin{align*}
D_\rho^\l (z,s)
&=\irn (b_\infty+b(x))|\psi_\rho(z,s)|^{p+1}
\\
&\ge [(1-s)^{p+1} +s^{p+1}]b_\infty
+p[(1-s)^{p}s+(1-s)s^p]b_\infty \e_{\rho }.
\end{align*}
Hence, by a Taylor expansion,
\[
\frac{N_\rho^\l (z,s)}{(D_\rho^\l (z,s))^{\frac2{p+1}}}
\le \frac 1{b_\infty^{\frac2{p+1}}}
\left(
\frac{[(1-s)^2+s^2]{\bf M}}{ [(1-s)^{p+1} +s^{p+1}]^{\frac2{p+1}}}
+2 \g (s){\bf M}\e_{\rho}+o(\e_{\rho})
\right),
\]
where
\[
\g (s)=\frac{(1-s)s}{ [(1-s)^{p+1} +s^{p+1}]^{\frac2{p+1}}}
\left(
1-\frac p{p+1}\frac{(1-s)^2+s^2}{ (1-s)^{p+1} +s^{p+1}}[(1-s)^{p-1} +s^{p-1}]
\right).
\]
Since $\g (1/2)<0$, there exists $\mathcal{I}_\frac 12$, neighborhood of $1/2$, such that $\g (s)<c<0$ for all $t\in \mathcal{I}_\frac 12$. Therefore, for $\rho$ 
large enough,
\[
\max\left\{
\frac{N_\rho^\l (z,s)}{(D_\rho^\l (z,s))^{\frac2{p+1}}}\ 
 \Big\vert \ z\in \Sigma, s \in \mathcal{I}_\frac 12
\right\}
\le \frac{2^\frac{p-1}{p+1}{\bf M}+2 c{\bf M}\e_{\rho}+o(\e_{\rho})}{b_\infty^{\frac2{p+1}}}
<2^\frac{p-1}{p+1}b_\infty^{-\frac2{p+1}}{\bf M}.
\]
On the other hand,
\begin{multline*}
\lim_{\rho\to +\infty}
\max\left\{
\frac{N_\rho^\l (z,s)}{(D_\rho^\l (z,s))^{\frac2{p+1}}}\ 
 \Big\vert \ z\in \Sigma, s \in [0,1]\setminus\mathcal{I}_\frac 12
\right\}
\\
\qquad \qquad
\le b_\infty^{-\frac2{p+1}}{\bf M}
\max\left\{
\frac{[(1-s)^2+s^2]}{ [(1-s)^{p+1} +s^{p+1}]^{\frac2{p+1}}}\ 
 \Big\vert \  s \in [0,1]\setminus\mathcal{I}_\frac 12
\right\}
<2^\frac{p-1}{p+1}b_\infty^{-\frac2{p+1}}{\bf M}.
\end{multline*}
Hence for $\rho $ sufficiently large,
\[
\max_{\Sigma\times[0,1]}\frac{N_\rho^\l (z,s)}{(D_\rho^\l (z,s))^{\frac2{p+1}}}<2^\frac{p-1}{p+1}b_\infty^{-\frac2{p+1}}{\bf M},
\]
so, recalling \eqref{bfm}, we conclude that
\[
\mathcal{T}_\rho^\l =
\max_{\Sigma\times[0,1]}
\Il(\Psi_\rho(z,s))
<\dfrac{p-1}{2(p+1)}
\left[2^\frac{p-1}{p+1}b_\infty^{-\frac2{p+1}}{\bf M}\right]^{\frac{p+1}{p-1}}
=2m_\infty.
\]
\end{proof}

\begin{lemma}\label{le:max<B}
Let assumptions of Lemma \ref{le:mB} hold. Then for large $\rho$,
\[
\mathcal{S}_\rho^\l :=
\max_{\Sigma}
\Il(\Psi_\rho(z,0))
<\mathcal{B}_0^\l.
\]
\end{lemma}

\begin{proof}
By \eqref{bwdef}, \eqref{bw} and \eqref{bfm}, for $\rho$ sufficiently large, we have  
\begin{align*}
\Il(\Psi_\rho(z,0))
&=\dfrac{p-1}{2(p+1)}
\left[
\dfrac{\dis \|\bw_{\rho z}\|^2+\l \irn a(x)\bw_{\rho z}^2}
{\dis \left( \irn (b_\infty+b(x))|\bw_{\rho z}|^{p+1}\right)^{\frac{2}{p+1}}}\right]^{\frac{p+1}{p-1}}
\\
&=\dfrac{p-1}{2(p+1)}
\left[
b_\infty^{-\frac2{p+1}}{\bf M}+o_\rho(1)\right]^{\frac{p+1}{p-1}}=m_\infty+o_\rho(1).
\end{align*}
Then the conclusion follows by Lemma \ref{le:mB}.
\end{proof}

\begin{proof}[Proof of Theorem \ref{th:seconda}]
Let $\l ^*$ be the number defined in \eqref{l*}. 
By Proposition \ref{pr:=}, $\l ^*\in \R^+$. Then, if $\l <\l ^*$, the relation  $m_\l <m_\infty$ holds and  $m_\l$ is achieved by  
Proposition \ref{pr:22'}.
\\
Let us suppose, now,  $\l >\l^*$. In this case, Propositions \ref{pr:mbar} and \ref{pr:l*} imply that $m_\l =m_\infty$, $m_\l $ is not achieved and 
the problem cannot be solved by minimization. However we are now going to prove that a solution of \eqref{eql} having energy greater than $m_\infty$ exists,
for all $\l>\l^*$.
\\
For any $c\in \R$, we set $I_\l^{c}:=\{u\in \Ne_\l : \Il(u)\le c\}$.
\\
By Lemmas \ref{le:B<=max}, \ref{le:M2m} and \ref{le:max<B}, the following chain of inequality holds
\[
m_\infty\le \mathcal{S}_\rho^\l<\mathcal{B}_0^\l\le \mathcal{T}_\rho^\l<2m_\infty.
\]
The argument is then completed showing that there exists $c^*\in [\mathcal{B}_0^\l, \mathcal{T}_\rho^\l]$ which is a critical level  of ${I_\l}_{\vert \Ne_\l}$.
In fact, in the opposite case, since $[\mathcal{B}_0^\l,\mathcal{T}_\rho^\l]\subset (m_\infty,2m_\infty)$, a positive number $\d >0$ and a continuous function 
$\eta:I_\l^{\mathcal{T}_\rho^\l} \to I_\l^{\mathcal{B}_0^\l-\d} $ can be found such that $\mathcal{B}_0^\l -\d >\mathcal{S}_\rho^\l$ and $\eta(u)=u$ for all
$u\in I_\l^{\mathcal{B}_0^\l-\d} $. Then
\beq\label{notin}
0\notin \b \circ \eta \circ \Psi_\rho(\Sigma\times [0,1]).
\eeq
On the other hand, since $\Psi_\rho(\Sigma\times\{0\})\subset I_\l^{\mathcal{S}_\rho^\l}$,  $\b \circ\eta \circ \Psi_\rho(\Sigma\times \{0\})$ is homotopically equivalent
to $\rho \Sigma$ in $\RN \setminus \{0\}$ and this implies
\[
0\in \b \circ \eta \circ \Psi_\rho(\Sigma\times [0,1]),
\]
which contradicts \eqref{notin}.
\\
When $\l =\l^* $, $m_{\l ^*}=m_\infty$ and either it is achieved or, if not, the arguments used for $\l>\l^*$ apply. Thus for $\l =\l^* $ too a solution of \eqref{eql} exists.
\\
Finally, since for any $\l \in \R^+$, we find a solution $u_\l$ of \eqref{eql} with $\Il(u_\l)<2m_\infty$, by Lemma \ref{le:24ter} we can assert that $u_\l$ does 
not change sign, so we can assume that it is positive.
\end{proof}

\begin{proof}[Proof of Theorem \ref{th:rad2}]
Set
\beq\label{531}
m_{\l,{\rm r}}:=\inf\{\Il(u): u\in \Ne_\l\cap \Hr\}
\eeq
clearly, for all $\l \ge 0$,
\[
0<m_\l \le m_{\l,{\rm r}}.
\]
Since $\Hr$ embeds compactly in $L^{p+1}(\RN)$, the infimum in \eqref{531} is actually a minimum to which there corresponds a nontrivial solution
of \eqref{eql}, by the Palais symmetric criticality principle.
Now, let us show that 
\[
\lim_{\l \to +\infty}m_{\l,{\rm r}}=+\infty.
\]
Let us observe that, arguing as in Proposition \ref{pr:mono}, we can show that the map $\l\in \R^+ \mapsto m_{\l,{\rm r}}$ is monotone non-decreasing.
\\
Assume, now,  by contradiction that a diverging sequence of numbers $(\l_n)_n$, a sequence of functions $(u_n)_n$ 
and a positive constant $C$ exist such that, for all $n\in \N$, $u_n \in \Ne_{\ln}\cap \Hr$ and $\Iln(u_n)=m_{\l_n,{\rm r}}\le C$.
Thus, by Lemmas \ref{le:un<c1} and \ref{le:un<c2}, $(u_n)_n$ is bounded in $\H$; furthermore
a positive constant $c$ and a sequence $(y_n)_n\subset \RN$, with $|y_n|\to +\infty$, must exist for which 
\beq\label{541}
\int_{B_1(y_n)}u_n^2\ge c>0.
\eeq
This last fact brings to a contradiction because, by the radial symmetry if $u_n$, \eqref{541} implies that $\|u_n\|\to +\infty$, as $n\to +\infty$.
\\
Therefore we get the conclusion just observing that for the solutions $u_\l $ whose existence has been stated in 
Theorems  \ref{th:<} and \ref{th:seconda}, whatever $\l \in \R^+$, the relation  $\Il(u_\l)< 2m_\infty$ holds. So, being  $m_{\l,{\rm r}}>2m_\infty$, for large $\l>0$, 
the existence of at least two distinct positive solutions of 
\eqref{eql} follows.
\end{proof}

%
%
%
%
%
%
%
%
%
%

\end{document}